\documentclass[12pt,a4paper]{amsart}
\usepackage[a4paper,centering]{geometry}
\usepackage{authoraftertitle,enumerate,fancyhdr,perpage}

\usepackage{mathptmx}
\DeclareMathAlphabet{\mathcal}{OMS}{cmsy}{m}{n} % do not change mathcal symbols

\theoremstyle{plain}
\newtheorem{theorem}{Theorem}
\newtheorem{proposition}{Proposition}
\newtheorem{lemma}{Lemma}
\newtheorem{corollary}{Corollary}
\newtheorem*{conjecture}{Conjecture}

\def\liebrack  {\ensuremath{[\,\cdot\, , \cdot\,]}}

\MakePerPage[2]{footnote}

\DeclareMathOperator{\Ad}{Ad}
\DeclareMathOperator{\Aut}{Aut}

\DeclareMathOperator{\Cent}{Cent}
\DeclareMathOperator{\Der}{Der}
\DeclareMathOperator{\End}{End}
\DeclareMathOperator{\GF}{\mathsf{GF}}
\DeclareMathOperator{\Hom}{Hom}
\DeclareMathOperator{\HomLie}{HomLie}
\DeclareMathOperator{\id}{id}
\DeclareMathOperator{\im}{Im}
\DeclareMathOperator{\Ker}{Ker}

\begin{document}

\title{When Hom-Lie structures form a Jordan algebra}
\author{Pasha Zusmanovich}
\address{University of Ostrava, Czech Republic}
\email{pasha.zusmanovich@osu.cz}
\date{First written September 9, 2021; last minor revision June 4, 2022}
\thanks{J. Algebra Appl., to appear}

\begin{abstract}
We are concerned with the question when Hom-Lie structures on a Lie algebra are
closed with respect to the Jordan product. Somewhat unexpectedly, this leads us
to certain questions connected with the Yang--Baxter equation, and with 
decomposition of a Lie algebra into the sum of subalgebras with given 
properties.
\end{abstract}

\vspace*{-0.8cm}

\maketitle

\pagestyle{fancy}
\setlength{\headheight}{0.5cm}

% the title is long:
%\fancyhead[L]{\footnotesize P. Zusmanovich \hfill \MyTitle \hfill}
%\fancyhead[C]{}

% the title is short:
\fancyhead[L]{\footnotesize P. Zusmanovich}
\fancyhead[C]{\footnotesize \MyTitle}

\fancyhead[R]{\footnotesize \thepage{}}
\fancyfoot[L,C,R]{}
\renewcommand{\headrulewidth}{0pt}

% to be adjusted
\vspace*{-1.0cm}
%%%%%%%%%%%%%%%%%%%%%%%%%%%%%%%%%%%%%%%%%%%%%%%%%%%%%%%%%%%%%%%%%%%%%%%

\section*{Introduction}

Recall that a \emph{Hom-Lie structure} on a Lie algebra $L$ is a linear map
$\varphi: L \to L$ satisfying the \emph{Hom-Jacobi equation}:
\begin{equation}\label{hom-jac}
[[x,y],\varphi(z)] + [[z,x],\varphi(y)] + [[y,z],\varphi(x)] = 0
\end{equation}
for any $x,y,z \in L$. The interest in such structures arose in the more general
context of so-called \emph{Hom-Lie algebras}, which are algebras with 
multiplication $\liebrack$ and endomorphism $\varphi$ satisfying the Hom-Jacobi
equation (\ref{hom-jac}). During the last decade there was a surge of interest
in such algebras; for the history and motivations see \cite{makhlouf}, 
\cite{homlie}, \cite{ado-homlie}, \cite{chin2}, \cite{xie-liu}, and references 
therein.

The set of all Hom-Lie structures on a Lie algebra $L$, denoted by $\HomLie(L)$,
obviously forms a vector space. Hom-Lie structures on simple finite-dimensional
Lie algebras over a field of zero characteristic, as well as on simple graded 
Lie algebras of finite growth (which are exhausted by loop algebras, untwisted 
or twisted; Lie algebras of Cartan type; and the Witt algebra) were described in
\cite{chin2} and \cite{xie-liu}. There, an interesting observation was made: on
these algebras, the space of all Hom-Lie structures is closed with respect to 
the anticommutator; that is, for any two Hom-Lie structures 
$\varphi, \psi \in \HomLie(L)$ on a Lie algebra $L$ from these classes, 
\begin{equation}\label{eq-jord}
\frac12 \big(\varphi \circ \psi + \psi \circ \varphi\big) \in \HomLie(L), 
\end{equation}
where $\circ$ denotes the composition of linear maps. In other words, 
$\HomLie(L)$ with respect to the anticommutator forms a (special) Jordan 
algebra. This was proved by case-by-case computations: for most of the algebras
$L$ from these classes, $\HomLie(L)$ coincides with the one-dimensional space 
$K \id_L$ consisting of scalar multiples of the identity map $\id_L$, and thus 
(\ref{eq-jord}) is satisfied trivially; in the nontrivial cases (mostly related
to $\mathfrak{sl}(2)$ and the Witt algebra), the validity of (\ref{eq-jord}) was
established by direct verification.

In \cite{homlie} we have provided further examples, in the class of current and 
Kac-Moody Lie algebras, for which $\HomLie(L)$ forms a Jordan algebra, and other
examples for which it does not. A natural question arises: for which Lie 
algebras this is true? To which degree this is a common phenomenon? 

In this note we show that unless the space of Hom-Lie structures is trivial, 
this phenomenon is rare, at least in the classes of ``interesting'' (i.e., 
simple and close to them) Lie algebras. In \S \ref{sec-2} we show that if the 
space of Hom-Lie structures is closed with respect to the anticommutator, then 
it either satisfies, as a Jordan algebra, some restrictive properties, or the 
underlying Lie algebra satisfies another (Lie-algebraic) restrictive properties,
dubbed by us the properties $\diamondsuit$ and $\heartsuit$. We focus mainly on
finite-dimensional algebras, so our results do not imply automatically all aforementioned results from 
\cite{chin2} and \cite{xie-liu}; but at least in the finite-dimensional case, in
\S \ref{sec-suit}, we sketch a uniform proof without going to case-by-case 
computations. In passing, in \S \ref{sec-ybe} we reformulate the Hom-Jacobi 
equation in terms of another equation which, in its turn, is related to the 
classical Yang-Baxter equation; in \S \ref{sec-witt} we discuss Hom-Lie 
structures on generalized Witt algebras; and in \S \ref{sec-suit} we discuss the
properties $\diamondsuit$ and $\heartsuit$; the property $\diamondsuit$ can be 
considered in the context which attracted a considerable attention in the 
literature: structure of Lie algebras decomposable into the vector space direct
sum of subalgebras with given properties. In the last \S \ref{sec-hom-jord} we
speculate about the possibility to replace in the considerations above 
``Jordan'' by ``Hom-Jordan''.

Our notation and conventions are mostly standard. The ground field $K$ is
assumed to be arbitrary, of characteristic $\ne 2$, unless specified otherwise.
By $\overline K$ is denoted the algebraic closure of $K$. All $\Hom$'s are 
understood in the category of vector spaces over $K$. Occasionally we will use 
the notion of the \emph{plus algebra} of an algebra $A$, denoted by $A^{(+)}$; 
this is the algebra defined on the same vector space $A$, with multiplication defined by the anticommutator of the initial 
multiplication in $A$: $a*b = \frac 12(ab + ba)$.

\section{Connection with the classical Yang--Baxter equation}\label{sec-ybe}

\begin{lemma}\label{lemma-1}
For any Lie algebra $L$, a linear map $\varphi: L \to L$ is a Hom-Lie structure
on $L$ if and only if the bilinear map $F_\varphi: L \times L \to L$ defined by
\begin{equation}\label{eq-f}
F_\varphi(x,y) = [\varphi(x),y] + [x,\varphi(y)]
\end{equation}
satisfies the equation
\begin{equation}\label{eq-F}
[F_\varphi(x,y),z] + [F_\varphi(z,x),y] + [F_\varphi(y,z),x] = 0
\end{equation}
for any $x,y,z \in L$.
\end{lemma}

\begin{proof}
Substituting (\ref{eq-f}) to (\ref{eq-F}) and rearranging terms, we get
\begin{align*}
    &[[\varphi(x),y],z] + [[z,\varphi(x)],y]   \\
+\> &[[\varphi(y),z],x] + [[x,\varphi(y)],z]   \\
+\> &[[\varphi(z),x],y] + [[y,\varphi(z)],x] 
= 0 .
\end{align*}
Using the Jacobi identity, the latter equality is equivalent to
$$
- [[y,z],\varphi(x)] - [[z,x],\varphi(y)] - [[x,y],\varphi(z)] = 0 ,
$$
which is exactly the Hom-Jacobi equation.
\end{proof}

The equation (\ref{eq-F}) is remarkable. Recall that a linear map 
$\varphi: L \to L$ on a Lie algebra $L$ is called an R-matrix if the bracket 
$$
[x,y]_R = \frac 12 \big([\varphi(x),y] + [x,\varphi(y)]\big)
$$
defines a new Lie algebra structure on $L$, i.e., satisfies the Jacobi identity
(for the definitions and facts related to R-matrices and Yang-Baxter equations 
mentioned in this paragraph, we refer to the survey \cite{rs}, \S 2). It is 
known that $\varphi$ is an R-matrix if and only if the bilinear map 
$$
F_\varphi(x,y) = 
[\varphi(x),\varphi(y)] - \varphi\big([\varphi(x),y] + [x,\varphi(y)]\big)
$$
satisfies the equation (\ref{eq-F}). In this situation the equation (\ref{eq-F}) is 
``usually replaced'' by the mere $F_\varphi(x,y) = 0$, or by 
$F_\varphi(x,y) = -[x,y]$, which constitutes, respectively, the classical 
Yang-Baxter equation, and the modified classical Yang-Baxter equation. Quite 
surprisingly, the equation (\ref{eq-F}) in the case of skew-symmetric (or, more
generally, arbitrary bilinear) $F$ was not studied systematically on various 
``interesting'' classes of Lie algebras. We suggest that such systematic study 
may justify that ``usual replacement'', and, at the same time, generalize all 
computations of Hom-Lie structures done so far (as, given (\ref{eq-f}) and
knowing $F_\varphi$, it is fairly easy to infer $\varphi$). Probably, it will be
relevant also in other situations (the equation (\ref{eq-F}) also arises in 
questions related to Hochschild cohomology of the smash product of a symmetric 
algebra and a group acting on it -- so-called orbifold algebras -- see 
\cite[\S 4]{orbifold} and references therein). On the other hand, 
\emph{symmetric} solutions of the equation (\ref{eq-F}) were studied in 
\cite{benkart} in the context of yet another, totally unrelated algebraic 
problem (determination of Lie-admissible third power-associative algebras). 

The equation (\ref{eq-F}) can be also interpreted as a binary extension of the
Hom-Lie equation (\ref{hom-jac}), where the univariate ``twisting'' map is 
replaced by a two-variate one, fitted into the Jacobi-like identity. Viewed this
way, Lemma \ref{lemma-1} provides an elementary, but interesting connection 
between ``unary'' and ``binary'' Hom-Lie structures on the same Lie algebra.

\medskip

\begin{lemma}\label{l-e}
For any Lie algebra $L$ the following are equivalent:
\begin{enumerate}[\upshape(i)]
\item $\HomLie(L)$ is closed with respect to the anticommutator;

\item For any $\varphi \in \HomLie(L)$, $\varphi^2 \in \HomLie(L)$;

\item
For any $\varphi \in \HomLie(L)$ and any polynomial $f \in K[t]$, 
$f(\varphi) \in \HomLie(L)$;

\item
For any $\varphi, \psi \in \HomLie(L)$, the bilinear map 
$F_{\varphi,\psi}: L \times L \to L$ defined as 
\begin{equation*}
F_{\varphi,\psi}(x,y) = [\varphi(x),\psi(y)] + [\psi(x),\varphi(y)]
\end{equation*}
satisfies the equation (\ref{eq-F}).
\end{enumerate}
\end{lemma}

\begin{proof}
(i) $\Rightarrow$ (ii). In the condition (\ref{eq-jord}), set $\psi = \varphi$.

(ii) $\Rightarrow$ (i). Linearize: replace $\varphi$ by $\varphi + \psi$.

(i) $\Rightarrow$ (iii). Follows from the fact that for any nonnegative integer
$n$, the $n$-fold anticommutator of a map $\varphi$ with itself coincides with 
$\varphi^n$.

(iii) $\Rightarrow$ (ii). Obvious.

(i) $\Leftrightarrow$ (iv). Straightforward computation, like in the proof of
Lemma \ref{lemma-1}.
\end{proof}

\section{Consequences of closedness of Hom-Lie structures with respect to the
anticommutator}\label{sec-2}

For the convenience of the exposition, we will fix from the very beginning the
consequences we will arrive at. These are two conditions imposed on a Lie 
algebra $L$:

\begin{enumerate}
\item[($\diamondsuit$)]
$L$ is decomposed into the direct sum of vector spaces $L = A \oplus B$ such
that $A,B \ne 0$, $[[A,A],B] = 0$, and $[[B,B],A] = 0$.

\item[($\heartsuit$)]
There are nonzero subspaces $A,B$ of $L$ such that $A \subseteq B$, 
$\dim A + \dim B = \dim L$, $[[A,A],B] = 0$, and $[[B,B],A] = 0$.
\end{enumerate}

\bigskip

We failed to find a reference in the literature to the following elementary
linear algebraic fact.

\begin{lemma}\label{lemma-la}
If $\varphi$ is a linear map of a finite-dimensional vector space over a perfect
field $K$, then there is a polynomial $f(t)$ with coefficients in $K$ such that
$f(\varphi)$ is an idempotent map of the same rank as $\varphi$.
\end{lemma}

\begin{proof}
1. Let $\varphi$ be nondegenerate. Since the free term of the characteristic 
polynomial $\chi(t)$ of $\varphi$ is equal to $\pm \det(\varphi)$, the required 
polynomial is $\frac{1}{\det \varphi}\chi(t) \mp 1$.

2. Let $\varphi$ be semisimple. Then the underlying vector space is decomposed
as the direct sum $\im(\varphi) \oplus \Ker(\varphi)$, the restriction of 
$\varphi$ on $\im(\varphi)$ is nondegenerate, and the rank of $\varphi$ is equal
to $\dim(\im(\varphi))$. By step 1, there is a polynomial $f(t)$ such that 
$f(\varphi)$ is the identity map on $\im(\varphi)$. Since $f(\varphi)$ acts 
trivially on $\Ker(\varphi)$, $f(t)$ will be a required polynomial.

3. In the general case, the semisimple component in the Jordan-Chevalley 
decomposition of $\varphi$ is equal to $f(\varphi)$ for a certain polynomial 
$f(t)$ (\cite[Chap. I, \S 8, Th\'eor\`eme 7]{chevalley}). By step $2$, there is
a polynomial $g(t)$ such that $g(f(\varphi))$ is an idempotent map of the same 
rank as $f(\varphi)$. Since the rank of $f(\varphi)$ is equal to the rank of
$\varphi$, $g(f(t))$ will be a required polynomial.
\end{proof}

\begin{lemma}\label{lemma-equiv}
For any finite-dimensional Lie algebra $L$ the following are equivalent:
\begin{enumerate}[\upshape(i)]
\item 
$L$ has an idempotent Hom-Lie structure, different from zero and from the 
identity map; 
\item
$L$ satisfies the property $\diamondsuit$.
\end{enumerate}
\end{lemma}

\begin{proof}
As any idempotent map $\varphi$ can be represented in a suitable basis by a 
diagonal matrix with $1$ and $0$ on the diagonal, we have a direct sum 
decomposition $L = A \oplus B$ such that $\varphi$ is the identity map on $A$ 
and the zero map on $B$. Then the conditions $[[A,A],B] = 0$ and $[[B,B],A] = 0$
are equivalent to the validity of the Hom-Jacobi equation.
\end{proof}

\begin{proposition}\label{lemma-f}
Let $L$ be a finite-dimensional Lie algebra over a perfect field, such that 
$\HomLie(L)$ is closed with respect to the anticommutator. Then one of the 
following holds:
\begin{enumerate}[\upshape(i)]
\item\label{i-11}
$\HomLie(L)$ is a Jordan algebra in which every element is either invertible or
nilpotent;
\item\label{i-22}
$L$ satisfies the equivalent conditions of Lemma \ref{lemma-equiv}.
\end{enumerate}

Moreover, if the ground field $K$ is algebraically closed, then the condition 
{\rm (i)} can be replaced by the condition
\begin{enumerate}[\upshape(i)$^\prime$]
\item
$\HomLie(L)$ is isomorphic to the semidirect sum of $K$ and a nilpotent algebra.
\end{enumerate}

\end{proposition}

\begin{proof}
Assume that the condition (\ref{i-11}) is not satisfied. Pick a Hom-Lie 
structure $\varphi$ on $L$ which is not invertible and is not nilpotent as an 
element of the Jordan algebra $\HomLie(L)$. Since $\HomLie(L)$ is a special 
Jordan algebra, the invertibility and nilpotency in the Jordan sense coincides, 
respectively, with invertibility and nilpotency in the associative sense; that 
is, the rank of $\varphi$ is strictly between $0$ and $\dim L$. By 
Lemma~\ref{l-e}, $f(\varphi)$ is a Hom-Lie structure on $L$ for any polynomial 
$f(t)$, and then by Lemma~\ref{lemma-la}, $L$ has a nontrivial idempotent 
Hom-Lie structure.

The Jordan algebra $\HomLie(L)$ is isomorphic to the semidirect sum of a 
semisimple algebra and the nilpotent radical. The semisimple part is isomorphic
to the direct sum of simples; if the sum contains more than one summand, then 
the unit of each summand is an idempotent different from the unit in the whole
$\HomLie(L)$, and them by Lemma \ref{lemma-equiv}, $L$ satisfies the property 
$\diamondsuit$. Therefore, in the condition (i) we may assume that $L$ is
isomorphic to the semidirect sum of a simple and a nilpotent algebra; but
if the ground field $K$ is algebraically closed, any finite-dimensional simple 
Jordan algebra with all elements either invertible or nilpotent, is isomorphic 
to $K$.
\end{proof}

\begin{proposition}\label{prop2}
Let $L$ be a finite-dimensional Lie algebra such that $\HomLie(L)$ is closed
with respect to the anticommutator. Then one of the following holds:
\begin{enumerate}[\upshape(i)]
\item\label{i-1}
$\HomLie(L)$ is a semisimple Jordan algebra without nonzero nilpotent elements;
\item\label{i-2}
$L$ satisfies the property $\heartsuit$.
\end{enumerate}

Moreover, if the ground field $K$ is algebraically closed, then the condition 
{\rm (\ref{i-11})} can be replaced by the condition
\begin{enumerate}[\upshape(i)$^\prime$]
\item
$\HomLie(L)$ is isomorphic to the direct sum of several copies of $K$.
\end{enumerate}
\end{proposition}

\begin{proof}
If $\HomLie(L)$ does not contain nonzero nilpotent elements, then its radical 
is zero, and hence it is a semisimple Jordan algebra.

If $\HomLie(L)$ contains a nonzero nilpotent element, then, raising it to the
appropriate power, and utilizing Lemma \ref{l-e}, we can find a nonzero Hom-Lie
structure $\varphi$ on $L$ such that $\varphi^2 = 0$. Set $A = \im(\varphi)$ and
$B = \Ker(\varphi)$. The equality $[[B,B],A] = 0$ follows from the Hom-Jacobi 
equation. Since $A \subseteq B$, the latter equality implies also 
$[[B,A],A] = 0$, and by the Jacobi identity we get $[[A,A],B] = 0$.

The Jordan algebra $\HomLie(L)$, being semisimple, is isomorphic to the direct 
sum of simple algebras, and if the ground field $K$ is algebraically closed, 
then each simple finite-dimensional Jordan algebra without nonzero nilpotent 
elements is isomorphic to $K$.
\end{proof}

\begin{corollary}\label{cor-algcl}
Let $L$ be a finite-dimensional Lie algebra over an algebraically closed 
field $K$, such that $\HomLie(L)$ is closed with respect to the anticommutator.
Then one of the following holds:
\begin{enumerate}[\upshape(i)]
\item $\HomLie (L) \simeq K$;
\item $L$ satisfies the property $\diamondsuit$ \hskip 1pt ;
\item $L$ satisfies the property $\heartsuit$ \hskip 1pt .
\end{enumerate}
\end{corollary}

\begin{proof}
If $L$ satisfies neither $\diamondsuit$ nor $\heartsuit$, then by 
Propositions~\ref{lemma-f} and \ref{prop2}, $L$ satisfies simultaneously the
respective conditions (i)$^\prime$ of those propositions, whence 
$\HomLie(L) \simeq K$.
\end{proof}

\medskip

Finally, let us indicate another strong consequence of closedness of Hom-Lie 
structures with respect to the anticommutator.

Recall that the \emph{centroid} of a Lie algebra $L$, denoted by $\Cent(L)$, is
the space of linear maps $\varphi: L \to L$ commuting with adjoint maps, i.e., 
satisfying the condition
$$
\varphi([x,y]) = [\varphi(x),y]
$$
for any $x,y \in L$. Centroid can be thought as the invariant submodule 
$\Hom(L,L)^L$ of the standard $L$-module $\Hom(L,L)$, with the $L$-action given
by the formula
\begin{equation}\label{eq-act}
(y \bullet \varphi)(x) = [\varphi(x),y] - \varphi([x,y]) ,
\end{equation}
where $x,y \in L$ and $\varphi \in \Hom(L,L)$.

It is clear that
\begin{equation}\label{eq-c}
\Cent(L) \subseteq \HomLie(L) .
\end{equation}

\begin{proposition}\label{prop-hom}
Let $L$ be a Lie algebra such that $\HomLie(L)$ is closed with respect to the
anticommutator. Then:
\begin{enumerate}[\upshape(i)]
\item
$L$ is homomorphically mapped to the Lie algebra 
$\Der(\HomLie(L))$;
\item
$\Aut(L)$ is homomorphically mapped to the group $\Aut(\HomLie(L))$.
\end{enumerate}
Here, in both cases, $\HomLie(L)$ is considered as a Jordan algebra with respect
to the anticommutator.
\end{proposition}

\begin{proof}
(i)
As noted in \cite[Lemma 1]{homlie}, $\HomLie(L)$ is a submodule of the 
$L$-module $\Hom(L,L)$. The $L$-action (\ref{eq-act}) is obviously compatible
with the product (\ref{eq-jord}), thus $L$ acts on the Jordan algebra 
$\HomLie(L)$ by derivations.

(ii) Similarly, $\HomLie(L)$ is a submodule of the $\Aut(L)$-module 
$\Hom(L,L)$, where the $\Aut(L)$-action is given by the conjugation 
$\Ad_\alpha: \varphi \mapsto \alpha^{-1} \circ \varphi \circ \alpha$ for 
$\varphi \in \HomLie(L)$ and $\alpha \in \Aut(L)$ (\cite[\S 1]{homlie}). This
action is obviously compatible with the product (\ref{eq-jord}), which yields 
the action of $\Aut(L)$ on the Jordan algebra $\HomLie(L)$.
\end{proof}

\begin{corollary}\label{cor-1}
Let $L$ be a simple Lie algebra such that $\HomLie(L)$ is closed with respect to
the anticommutator. Then either $\HomLie(L) = \Cent(L)$, or $L$ is isomorphic to
a subalgebra of $\Der(\HomLie(L))$.
\end{corollary}

\begin{proof}
Since $L$ is simple, the kernel of the homomorphism from 
Proposition~\ref{prop-hom}(i) either coincides with the whole $L$, or is zero. 
In the first case, $L$ acts on $\HomLie(L)$ trivially, i.e., 
$\HomLie(L) \subseteq \Cent(L)$, and due to (\ref{eq-c}), the equality holds. In
the second case, the homomorphism is an embedding.
\end{proof}

\section{Witt algebras, finite- and infinite-dimensional}\label{sec-witt}

To see that the proofs of Propositions \ref{lemma-f} and \ref{prop2} do not work
in the infinite-dimensional case, let us turn to (generalized) Witt algebras. 
Let $G$ be a subgroup of the additive group of the ground field $K$, and $W_G$ 
is the Lie algebra linearly spanned by elements $e_\alpha$, $\alpha \in G$, with
the bracket
$$
[e_\alpha , e_\beta] = (\beta - \alpha) e_{\alpha + \beta} .
$$

Specializing $G$ to various particular cases, we get various instances of Witt 
algebras. Thus, if $K$ is of characteristic zero, and $G = \mathbb Z$, we get 
the (two-sided) infinite-dimensional Witt algebra, and if $K$ is of 
characteristic $p > 0$, containing the field $G = \GF(p^n)$ (isomorphic, as an 
additive group, to $(\mathbb Z/p\mathbb Z)^n)$, we get the Zassenhaus algebra 
$W_1(n)$. The latter algebra has another realization as a $\mathbb Z$-graded Lie
algebra, with a basis $\{e_{-1}, e_0, e_1, \dots, e_{p^n-2}\}$, and 
multiplication
\begin{equation}\label{eq-zass}
[e_i,e_j] = \Big(\binom{i+j+1}{j} - \binom{i+j+1}{i}\Big) e_{i+j} .
\end{equation}

\begin{theorem}\label{th-w}
$\HomLie(W_G) \simeq K[G]$. A basis of $\HomLie(W_G)$ may be chosen to consist 
of ``shifts'' \linebreak $e_\alpha \mapsto e_{\alpha + \sigma}$ for a fixed 
$\sigma \in G$.
\end{theorem}

\begin{proof}
A verbatim repetition of reasonings in the proof of \cite[Theorem 3.2]{xie-liu},
which treats the case $G = \mathbb Z$.
\end{proof}

Therefore, the Hom-Lie structures on $W_G$ are closed with respect to 
composition (and thus with respect to the anticommutator), and form the 
commutative associative algebra isomorphic to the group algebra $K[G]$ (in the case 
$G = \mathbb Z$ this was already noted in \cite{xie-liu}).

\begin{proposition}\label{prop-w}
If $K$ is of characteristic zero, then the algebra $W_G$ satisfies neither the 
property $\diamondsuit$, nor the property $\heartsuit$.
\end{proposition}

\begin{proof}
Assume the contrary. Any nonzero abelian subalgebra of $W_G$ 
is one-dimensional (see, e.g., \cite[Corollary(a)]{kubo}). Hence the condition
$[[A,A],B] = 0$ implies either $[A,A] = 0$, and hence $A$ is one-dimensional, or
$[A,A] = B$ is one-dimensional. As in \cite{kubo}, using the fact that $G$ can 
be ordered, it is easy to see that if $A$ is a subspace of $W_G$ such that 
$[A,A]$ is one-dimensional, then $A$ is two-dimensional. Therefore, in any case
$\dim A \le 2$. Similarly, $[[B,B],A] = 0$ implies $\dim B \le 2$, and hence $\dim W_G \le 4$, a contradiction.
\end{proof}

By contrast with the infinite-dimensional characteristic zero case, we have

\begin{proposition}
If $G$ is finite, then the algebra $W_G$ does not satisfy the property
$\diamondsuit$, and satisfies the property $\heartsuit$.
\end{proposition}

\begin{proof}
If $G$ is finite, then $K$ is necessarily of characteristic $p>0$, $G$ is the
additive group of $\GF(p^n)$ for some $n$, $W_G \simeq W_1(n)$, and by 
Theorem~\ref{th-w}, $\HomLie(W_G)$ is isomorphic to the reduced polynomial 
algebra $K[x_1,\dots,x_n]/(x_1^p,\dots,x_n^p)$. The latter algebra does not 
contain nontrivial idempotents, and by Lemma \ref{lemma-equiv}, $W_G$ does not 
satisfy the property $\diamondsuit$.

Looking at the realization (\ref{eq-zass}) of $W_1(n)$, and setting 
$A = \langle e_{p^n-2} \rangle$ and 
$B = \langle e_0, e_1, \dots, e_{p^n-2} \rangle$, we see that $W_1(n)$ satisfies
the property $\heartsuit$.
\end{proof}

We believe that the algebra $W_G$ does not satisfy the property $\diamondsuit$
in all cases, but the proof of this is lacking (infinite-dimensional Witt 
algebras over fields of positive characteristic seem to be more tricky).

Proposition \ref{prop-w} shows that $W_G$, in the case of zero characteristic, 
provides an infinite-dimensional counterexample to Propositions \ref{lemma-f} 
and \ref{prop2}. Indeed, using the fact that $G$ is ordered, it is easy to see 
that $K[G]$ does contain neither nontrivial idempotents, nor nontrivial 
nilpotent elements. Moreover, $K[G]$ satisfies neither the condition 
(\ref{i-11}) of Proposition~\ref{lemma-f}, nor the condition (i)$^\prime$ of 
Proposition~\ref{prop2} (but satisfies the condition (\ref{i-1}) of 
Proposition~\ref{prop2}).

\section{More on algebras satisfying the properties $\diamondsuit$ and 
$\heartsuit$}\label{sec-suit}

As an illustration of an application of 
Propositions~\ref{lemma-f}-\ref{prop-hom}, let us sketch the proof of the 
following 

\begin{theorem}
If $\mathfrak g$ is a central simple finite-dimensional Lie algebra over a 
field $K$ of characteristic zero such that $\HomLie(\mathfrak g)$ is closed with
respect to the commutator, then either $\HomLie(\mathfrak g) \simeq K$, or 
$\mathfrak g$ is $3$-dimensional.
\end{theorem}

We give merely a sketch of the proof as, first, as noted in the introduction, 
this result is not new, and follows from computation of Hom-Lie structures on 
these algebras in \cite{chin2} and \cite{xie-liu} (see also 
\cite[Theorem 1]{homlie}), and, second, the proof -- modulo the results of 
\S\ref{sec-2} -- is fairly elementary. However, we want to demonstrate how one
can achieve such sort of results without elaborate computations of Hom-Lie 
structures. Also, in our approach the exceptional $3$-dimensional case emerges 
in a quite interesting way.

Since the Hom-Jacobi equation (\ref{hom-jac}) is linear in $\varphi$, for any
Lie algebra $L$ over a field $K$ we have
$$ 
\HomLie (L \otimes_K \overline K) \simeq \HomLie(L) \otimes_K \overline K .
$$
Further, since the condition (\ref{eq-jord}) of closedness of Hom-Lie structures
with respect to the anticommutator is bilinear in $\varphi, \psi$, a Lie 
$K$-algebra satisfies this condition if and only if the Lie 
$\overline K$-algebra $L \otimes_K \overline K$ does. This allows us to use
the ``Weyl's unitary trick'' (see, e.g., \cite[Chap. IV, \S 7]{jacobson}), and
reduce the proof to the case of compact Lie algebras.

Compact Lie algebras possess many peculiar properties (they are closed with
respect to subalgebras, and have no nilpotent subalgebras and elements), and it
is fairly easy to prove -- using, for example, induction by dimension -- that
if a simple compact Lie algebra satisfies one of the properties $\diamondsuit$, 
$\heartsuit$, then it is isomorphic to the $3$-dimensional algebra 
$\mathfrak{su}(2)$. Now, we cannot use Corollary \ref{cor-algcl}, as we are not
over an algebraically closed field, but over $\mathbb R$, but we can use 
similar reasonings valid in the real case. Namely, by Propositions~\ref{lemma-f}
and \ref{prop2}, if $\mathfrak g \not\simeq \mathfrak{su}(2)$, then 
$\HomLie(\mathfrak g)$ is a semisimple Jordan algebra in which every nonzero 
element is invertible. Then $\HomLie(\mathfrak g)$ does not have nontrivial
idempotents, and hence is simple. Inspection of the list of simple real Jordan 
algebras (available, for example, in \cite[Appendix A]{real}) reveals that the 
only simple Jordan algebras with the required property are either the 
$1$-dimensional algebra $\mathbb R$, or the $4$-dimensional plus algebra 
$\mathbb Q^{(+)}$ of the real quaternion division algebra $\mathbb Q$. In the 
latter case, by Corollary \ref{cor-1}, $\mathfrak g$ is isomorphic to a 
subalgebra of $\Der(\mathbb Q^{(+)}) \simeq \mathfrak{su}(2)$, and hence is 
isomorphic to $\mathfrak{su}(2)$.

\medskip

What can be said about Lie algebras satisfying the property $\diamondsuit$ or
$\heartsuit$ in general? It seems that the exact description in the general case
could be difficult. However, the condition $\diamondsuit$ can be interpreted
as a generalization of the following condition: a Lie algebra is a vector space
sum (not necessarily direct) of two abelian subalgebras. It is easy to prove 
that such Lie algebras are metabelian (see, for example, 
\cite[Proposition 1.5]{kolman}). There is a vast body of literature devoted to 
generalizations and extensions of this situation: what happens when we impose 
various restrictions on summands (nilpotency, simplicity, etc.); see, for 
example, \cite{burde-intro} and references therein. The property $\diamondsuit$
can be thought as a generalization in another direction: in the decomposition 
$L = A \oplus B$ we no longer assume $[A,A] = 0$ and $[B,B] = 0$, but impose the
weaker conditions $[[A,A],B] = 0$ and $[[B,B],A] = 0$, without assuming that $A$, $B$ are 
necessarily subalgebras.

The property $\heartsuit$ seems to be more tricky; we conjecture that it is 
related to existence of subalgebras of codimension $1$ (recall that a simple
finite-dimensional Lie algebra with a subalgebra of codimension $1$ is a form
either of $\mathfrak{sl}(2)$, or of Zassenhaus algebra, see \cite{elduque} and
references therein).

\begin{conjecture}
Let $L$ be a finite-dimensional Lie algebra over an algebraically closed field 
of characteristic $\ne 2$, and $Rad(L)$ its solvable radical.
\begin{enumerate}[\upshape(i)]
\item
If $L$ satisfies the property $\diamondsuit$, then $L/Rad(L)$ is isomorphic to 
the direct sum of several copies of $\mathfrak{sl}(2)$.

\item
If $L$ satisfies the property $\heartsuit$, then $L/Rad(L)$ is isomorphic to the
direct sum of several copies of $\mathfrak{sl}(2)$ and $W_1(n)$.
\end{enumerate}
\end{conjecture}

As the properties $\diamondsuit$ and $\heartsuit$ are preserved under field 
extensions, the condition that the ground field is algebraically closed is 
immaterial here, and is added merely to avoid cumbersome formulations related to
forms of not necessarily central semisimple algebras.

\section{When Hom-Lie structures form a Hom-Jordan algebra?}\label{sec-hom-jord}

Here we briefly discuss the question posed by Sergei Silvestrov: if we are
dealing with Hom-algebras, wouldn't it be natural to replace in the question we
dealt with in this paper, ``Jordan'' by ``Hom-Jordan''? To properly interpret 
this question, we should replace the ordinary anticommutator (\ref{eq-jord}) in
the associative algebra $\End(L)$ of all linear maps of $L$ as a vector 
space, by its Hom-version; and for this, we need the notion of both 
Hom-associative and Hom-Jordan algebra.

According to the general idea, to get Hom-versions of identities from their 
standard counterparts, one should to ``twist'' them by a linear map, similarly 
how the identity (\ref{hom-jac}) is obtained from the Jacobi identity. In this 
way, a \emph{Hom-associative algebra} is an algebra $A$ with a binary multiplication 
$\>\cdot\>$, and a twisting linear map $\varphi: A \to A$, satisfying the 
Hom-version of the associative identity:
$$
(x \cdot y) \cdot \varphi(z) = \varphi(x) \cdot (y \cdot z) .
$$

As for \emph{Hom-Jordan algebras}, there are two versions of them in the 
literature: Makhlouf in \cite{makhlouf} defines them as commutative algebras
satisfying the identity
$$
(y \cdot x^2) \cdot \varphi^2(x) = (y \cdot \varphi(x)) \cdot \varphi(x^2) ,
$$
while Yau in \cite{yau} defines them as commutative algebras satisfying the 
identity
$$
(\varphi(y) \cdot x^2 ) \cdot \varphi^2(x) = 
(\varphi(y) \cdot \varphi(x)) \cdot \varphi(x^2) .
$$

These two definitions are different, and Yau argues, not without reason, that 
his definition is more ``correct'', as the plus algebra of a Hom-alternative 
algebra (whatever it is), is always Hom-Jordan in his sense, but not in 
Makhlouf's sense. For us, however, this difference is immaterial, as we are
concerned here exclusively with Hom-Jordan algebras of the form $A^{(+)}$ for 
Hom-\emph{associative} algebras $A$. As shown respectively in \cite{makhlouf} 
and \cite{yau} (and is easy to see), the plus algebra of a Hom-associative 
algebra is Hom-Jordan in both senses.

Now, as explained in \cite[\S 2]{ado-homlie}, the proper Hom-analog of the 
associative algebra $\End(V)$ of all linear maps on the vector space $V$, is the
Hom-associative algebra $\End(V)_\alpha$, with multiplication 
$$
\varphi \cdot_\alpha \psi = 
\alpha^{-1} \circ \varphi \circ \alpha \circ \psi \circ \alpha
$$
and the twisting map 
$\Ad_\alpha: \varphi \mapsto \alpha^{-1} \circ \varphi \circ \alpha$, where
$\varphi, \psi \in \End(V)$, and $\alpha: V \to V$ is a fixed invertible linear
map. Thus, the multiplication in the Hom-Jordan algebra $\End(V)_\alpha^{(+)}$ 
is defined by the anticommutator
\begin{equation}\label{eq-chi}
\varphi *_\alpha \psi = 
\frac 12 
\alpha^{-1} \circ (\varphi \circ \alpha \circ \psi + 
                   \psi \circ \alpha \circ \varphi) \circ \alpha
= 
\frac 12 \Ad_\alpha (\varphi \circ \alpha \circ \psi 
                   + \psi \circ \alpha \circ \varphi) .
\end{equation}

Therefore, the question can be formulated as follows: for which Lie algebras
$L$ and an invertible linear map $\alpha: L \to L$, for any two Hom-Lie 
structures $\varphi, \psi \in \HomLie(L)$, their $\alpha$-``twisted'' 
anticommutator, as defined in (\ref{eq-chi}), is also a Hom-Lie structure on 
$L$?

A particular, but, perhaps, more attractive variant of this question assumes 
$\alpha \in \Aut(L)$. In this case, by a result from \cite[\S 1]{homlie}, 
already mentioned in the proof of Proposition \ref{prop-hom}, $\HomLie(L)$ is
invariant under $\Ad_\alpha$, and hence $\Ad_\alpha$ in the formula 
(\ref{eq-chi}) can be dropped. Thus this version of the question reads: for 
which Lie algebras $L$ and $\alpha\in \Aut(L)$, for any two Hom-Lie 
structures $\varphi, \psi \in \HomLie(L)$, the map
$$
\frac 12 (\varphi \circ \alpha \circ \psi + \psi \circ \alpha \circ \varphi)
$$
is also a Hom-Lie structure on $L$? Note that this imposes a strong restriction 
on the automorphism $\alpha$: in particular, it itself should be a Hom-Lie 
structure on $L$.

\section*{Acknowledgement}

Thanks are due to Sergei Silvestrov for interesting discussions. This work was 
supported by the grant AP08855944 of the Ministry of Education and Science of 
the Republic of Kazakhstan.


\begin{thebibliography}{BCK}

\bibitem[BCK]{real} S. Ben Sa\"id, J.-L. Clerc, K. Koufany,
\emph{Conformally covariant bi-differential operators on a simple real Jordan 
algebra},
Intern. Math. Res. Notices 2020, no.8, 2287--2351.

\bibitem[Be]{benkart} G.M. Benkart, 
\emph{Power-associative Lie-admissible algebras}, 
J. Algebra \textbf{90} (1984), no.1, 37--58.

\bibitem[Bu]{burde-intro} D. Burde, 
\emph{Derived length and nildecomposable Lie algebras},
Scientific Bulletin of the ``Politehnica'' University of Timisoara 
\textbf{58} (2013), 15--24.

\bibitem[C]{chevalley} C. Chevalley, 
\emph{Th\'eorie des groupes de Lie. Tome II. Groupes alg\'ebriques},
Hermann, Paris, 1951.

\bibitem[E]{elduque} A. Elduque, 
\emph{On Lie algebras with a subalgebra of codimension one}, 
Lie Algebras, Madison 1987 (ed. G. Benkart, J.M. Osborn), 
Lect. Notes Math. \textbf{1373} (1989), 58--66.

\bibitem[FK]{orbifold} B. Foster-Greenwood, C. Kriloff,
\emph{Drinfeld orbifold algebras for symmetric groups},
J. Algebra \textbf{491} (2017), 573--610.

\bibitem[J]{jacobson} N. Jacobson, \emph{Lie Algebras}, 
Interscience Publ., 1962; reprinted by Dover, 1979.

\bibitem[Ko]{kolman} B. Kolman, \emph{Semi-modular Lie algebras}, 
J. Sci. Hiroshima Univ. Ser. A-I \textbf{29} (1965), no.2, 149--163.

\bibitem[Ku]{kubo} F. Kubo, \emph{A note on Witt algebras}, 
Hiroshima Math. J. \textbf{7} (1977), no.2, 473--477.

\bibitem[M]{makhlouf} A. Makhlouf, 
\emph{Hom-alternative algebras and Hom-Jordan algebras},
Intern. Electr. J. Algebra \textbf{8} (2010), 177--190.

\bibitem[MZ1]{homlie} A. Makhlouf, P. Zusmanovich, 
\emph{Hom-Lie structures on Kac--Moody algebras},
J. Algebra \textbf{515} (2018), 278--297.

\bibitem[MZ2]{ado-homlie} A. Makhlouf, P. Zusmanovich, 
\emph{Ado theorem for nilpotent Hom-Lie algebras},
Intern. J. Algebra Comp. \textbf{29} (2019), no.7, 1343--1365.

\bibitem[RS]{rs} A.G. Reyman, M.A. Semenov-Tian-Shansky, 
\emph{Group-theoretical methods in the theory of finite-dimensional integrable 
systems}, 
Dinamicheskie Sistemy - 7, VINITI, 1987, 119--194 (in Russian);
Dynamical Systems VII, Encyclopaedia of Mathematical Sciences, Vol. 16, 
Springer-Verlag, 1994, 116--225 (English translation).

\bibitem[XJL]{chin2} W. Xie, Q. Jin, W. Liu,
\emph{Hom-structures on semi-simple Lie algebras},
Open Math. \textbf{13} (2015), 617--630.

\bibitem[XL]{xie-liu} W. Xie, W. Liu,
\emph{Hom-structures on simple graded Lie algebras of finite growth},
J. Algebra Appl. \textbf{16} (2017), no.8, 1750154.

\bibitem[Y]{yau} D. Yau, 
\emph{Hom-Maltsev, Hom-alternative, and Hom-Jordan algebras},
Intern. Electr. J. Algebra \textbf{11} (2012), 177--217.

\end{thebibliography}
\end{document}